\documentclass[10pt]{article}
\usepackage[english,russian]{babel}
\usepackage{amssymb}
\usepackage{dsfont}
\usepackage{amsfonts}
\usepackage{amsmath,amsthm}

\DeclareMathOperator{\LipSh}{LipSh}
\DeclareMathOperator{\Int}{Int}
\DeclareMathOperator{\Rep}{Rep}
\DeclareMathOperator{\Per}{Per}
\DeclareMathOperator{\Ker}{Ker}
\DeclareMathOperator{\OrSh}{OrSh}
\DeclareMathOperator{\dist}{dist}
\DeclareMathOperator{\diag}{diag}

\begin{document}

\righthyphenmin = 2

\newcommand{\RR}{\mbox{${\bf R}$}}
\newcommand{\Cc}{\mbox{${\bf C}$}}
\newcommand{\Ff}{\mbox{${\cal F}$}}
\newcommand{\Ws}{\mbox{$W^s$}}
\newcommand{\Wu}{\mbox{$W^u$}}
\newcommand{\Cone}{\mbox{$\Cc^1$}}

\newcommand{\sref}[1]{(\ref{#1})}

\newtheorem{thm}{Теорема }
\newtheorem{lem}{Лемма }
\theoremstyle{definition}
\newtheorem{defin}{Определение }

\author{C. Б. Тихомиров}
\title{Внутренности множеств векторных полей со свойствами отслеживания,
соответствующими некоторым классам репараметризаций}
\date{}

\maketitle

\section{Abstract}

Изучается структура $\Cone$-внутренности множеств векторных полей,
обладающих различными видами свойства отслеживания. Показано, что
для случая липшицевого свойства отслеживания она совпадает с
множеством структурно устойчивых систем. В случае, если размерность
многообразия не превышает 3, аналогичный результат верен для
ориентированного свойства отслеживания.

\section{Введение}

Задача об отслеживании псевдотраекторий связана со следующим
вопросом: при каких условиях для любой псевдотраектории динамической
системы можно найти близкую к ней траекторию? Изучение данной задачи
было начато Д. В. Аносовым \cite{Ano} и P. Боуэном \cite{Bow}.
Современное состояние теории отслеживания в значительной степени
отражено в монографиях \cite{Pil, Palm}.

Отметим, что основное отличие задачи об отслеживании для потоков от
аналогичной задачи для дискретных динамических систем, порождаемых
диффеоморфизмами, состоит в репараметризации отслеживающих
траекторий.

Цель данной статьи -- описать структуру $\Cone$-внутренности
множеств векторных полей, обладающих теми или иными свойствами
отслеживания псевдотраекторий.

\section{Основные обозначения и результаты}

Пусть $M$ -- гладкое $n$-мерное замкнутое (т.е. компактное без края)
многообразие с римановой метрикой $\dist$. Обозначим через $\Ff(M)$
пространство гладких векторных полей на $M$ с $\Cone$-топологией.
Для векторного поля $X \in \Ff(M)$ будем обозначать через $\phi(t,
x)$ такую траекторию поля $X$, что $ \phi(0, x) = x$.
\begin{defin}
Пусть $d>0$. Будем называть \textit{$d$-псевдотраекторией} поля $X$
такое отображение $g: \RR \to M, $ что $ \dist(g(t + \tau), \phi(t,
g(\tau))) < d $ для $|t|<1$, $\tau \in \RR$.
\end{defin}

Введем понятие свойства отслеживания для потоков. Важнейшую роль в
свойстве отслеживания для потоков играют репараметризации.
\begin{defin} Назовем \textit{репараметризацией} такой возрастающий
гомеоморфизм $h: \RR \to \RR, $ что $h(0) = 0$. Для $a>0$ обозначим
через $\Rep(a)$ множество репараметризаций, удовлетворяющих
неравенству
$$
\left| \frac{h(t_1) - h(t_2)}{t_1 - t_2} - 1 \right| \leq a \quad
\mbox{для} \quad t_1, t_2 \in \RR, \quad t_1 \ne t_2.
$$
\end{defin}

\begin{defin}
Будем говорить, что поток $\phi$ обладает \textit{ориентированным
свойством отслеживания}, если по любому $\varepsilon > 0$ найдется
такое $d>0$, что для любой $d$-псевдотраектории $g$ можно указать
такие точку $p$ и репараметризацию $h$, что выполнено неравенство
\begin{equation}\notag
\dist(\phi(h(t), p), g(t)) < \varepsilon, \quad t \in \RR.
\end{equation}
\end{defin}

\begin{defin}
Будем говорить, что поток $\phi$ обладает \textit{липшицевым
свойством отслеживания}, если существуют $L_0, D_0 >0$ со следующим
свойством: для любых $d< D_0$ и $d$-псевдотраектории $g$ можно
указать такие точку $p$ и репараметризацию $h \in \Rep(L_0 d)$, что
выполнено неравенство
\begin{equation}\notag
\dist(\phi(h(t), p), g(t)) < L_0d, \quad t \in \RR.
\end{equation}
\end{defin}

Будем обозначать через $\OrSh$ и $\LipSh$ множества векторных полей,
обладающих ориентированным и липшицевым свойствами отслеживания,
соответственно. Кроме того, будем обозначать через $S$ множество
структурно устойчивых векторных полей, через  $T$ множество
векторных полей, у которых все точки покоя и замкнутые траектории
гиперболичны, и через $KS$ множество полей Купки-Смейла
\cite{PilRus}. Ясно, что $ \LipSh \subset \OrSh$.

Для любого множества $A \subset \Ff(M)$ будем через $\Int^1(A)$
обозначать \hbox{$\Cone$-внутренность} множества $A$. Для векторного
поля $X$ обозначим через $\Per(X)$ множество точек покоя и замкнутых
траекторий поля $X$. Для всякой гиперболической траектории $p \in
\Per(X)$ будем обозначать через $\Ws(p)$ и $\Wu(p)$ ее устойчивое и
неустойчивое многообразие, соответственно.

В \cite{Pil7}  показано, что $S \subset \LipSh.$ Так как множество
$S$ является \hbox{$\Cone$-открытым}, то $ S \subset
\Int^1(\LipSh)$. Основные результаты данной статьи таковы:
\begin{thm}\label{Lip}
$ S = \Int^1(\LipSh).$
\end{thm}
\begin{thm}\label{Dim3}
Если  $\dim M \leq 3$, то выполнено равенство $ S = \Int^1(\OrSh)$.
\end{thm}

\section{Доказательство теоремы \ref{Lip}}

Модифицируя технику, примененную в \cite{Pil2} для случая
диффеоморфизмов, легко доказать следующее утверждение:
\begin{lem}\label{Hyp}
$\Int^1(\OrSh) \subset T$.
\end{lem}

Ган доказал в \cite{Gan}, что $\Int^1(KS) = S$. Таким образом, для
доказательства теоремы \ref{Lip} осталось доказать следующее
утверждение:
\begin{lem}\label{Rest}
Пусть $X \in \Int^1(\LipSh)$ и $p, q \in \Per(X)$. Если $ r \in
\Wu(q) \cap \Ws(p)$, то $r$-- точка трансверсального пересечения
$\Wu(q)$ и $\Ws(p)$.
\end{lem}
\begin{proof}

Мы приведем доказательство этой леммы для наиболее трудного случая,
в котором $p$ и $q$ -- точки покоя. В остальных случаях, используя
методы, описанные в \cite{Pil2}, \cite{Pil3}, можно доказать
аналогичное утверждение для ориентированного свойства отслеживания:
\begin{lem}\label{OrSP}
Пусть $X \in \Int^1(\OrSh)$. Пусть $\gamma_1$ -- замкнутая
траектория поля $X$, а $\gamma_2 \in \Per(X)$. Пусть $r_0 \in
\Ws(\gamma_1) \cap \Wu(\gamma_2)$. Тогда $r_0$ -- точка
трансверсального пересечения $\Ws(\gamma_1)$ и $\Wu(\gamma_2)$.
\end{lem}
Нам понадобятся две элементарные технические леммы, доказательства
которых мы опустим.

Рассмотрим на плоскости $\RR^2$ поток $\varphi(t, x)$, порожденный
линейной автономной системой вида
\begin{equation}\notag
\dot{x} = \left(
            \begin{array}{cc}
              a & -b \\
              b & a \\
            \end{array}
          \right)x, \quad x \in \RR^2, \quad \mbox{где $a>0$ и $b \ne
          0$}.
\end{equation}
Для точки $x \in \RR^2 \setminus \{0\}$ будем обозначать через
$\arg(x)$ точку $ \frac{x}{|x|} \in S^1$.
\begin{lem}\label{Spiral}
Для любых $\varepsilon, L > 0$ существуют такие положительные числа
$T = T(\varepsilon, L)$ и $d_0 = d_0(\varepsilon, L)$, что если
$$
 d< d_0, \quad
 x_0, x_1 \in \RR^2, \quad |x_0| \geq d, \quad
 h(t) \in \Rep(Ld),
$$
\begin{equation}\label{Lip4}
|\varphi(t, x_0) - \varphi(h(t), x_1)| < Ld \quad \mbox{при} \quad t
\in [0, T],
\end{equation}
то выполнено неравенство $ |\arg(x_1) - \arg(x_0)| < \varepsilon. $
\end{lem}

Одномерным (и более простым) аналогом леммы \ref{Spiral} является
следующее утверждение, относящееся к дифференциальному уравнению $
\dot{x} = ax$ на прямой и к его потоку $\varphi(t, x) = x e^{at}$.
\begin{lem}\label{Dim1Spiral}
Для любых $\varepsilon, L > 0$ существуют такие положительные числа
$T = T(\varepsilon, L)$ и $d_0 = d_0 (\varepsilon, L)$, что если
$$
 d < d_0, \quad
 x_0, x_1 \in \RR, \quad |x_0|>d, \quad
 h(t) \in \Rep(Ld),
$$
и выполнено неравенство \sref{Lip4}, то выполнено неравенство
\begin{equation}\notag
\frac{|x_1 - x_0|}{|x_0|} < \varepsilon.
\end{equation}
\end{lem}

Приступим к доказательству леммы \ref{Rest}. Предположим противное:
пусть $r$ -- точка нетрансверсального пересечения $\Wu(q)$ и
$\Ws(p)$. В \cite{Pil2} и \cite{Pil3} показано, что в любой
$\Cone$-окрестности поля $X$ найдется такое поле $X'$, что $p$ и $q$
гиперболические точки покоя поля $X'$, $r$ -- точка
нетрансверсального пересечения $\Wu(q)$ и $\Ws(p)$ и поле $X'$
линейно в некоторых окрестностях $N_p$ и $N_q$ точек $p$ и $q$
соответственно. Поскольку $X \in \Int^1(\LipSh)$, то $X'$ можно
выбрать также принадлежащим $\Int^1(\LipSh)$. Для дальнейшей
простоты изложения переобозначим $X'$ через $X$, а поток,
порожденный $X'$, через $\phi$. В дальнейшем в ходе доказательства
мы будем несколько раз подобным образом возмущать поле $X$, оставляя
его в $\Int^1(\LipSh)$ и переобозначая новое поле через $X$, а поток
через $\phi$.

Отождествим окрестности $N_p$ и $N_q$ с пространством $\RR^n$.
Введем в $N_p$ и $N_q$ такие локальные координаты $(y, z)$ и $(\xi,
\eta)$, что $p$ и $q$ -- начала координат в $N_p$ и $N_q$, а матрицы
Якоби в этих координатах имеют вид (возможно для этого придется
возмутить поле $X$): $ DX(p) = \diag(A_p, B_p), $ где $
Re(\lambda_j) < 0$ для собственных чисел $A_p$, $Re(\lambda_j) > 0$
для собственных чисел $B_p$, и $ B_p = \diag (\lambda_1, \dots,
\lambda_{u_1}, D_1, \dots, D_{u_2}), $ где $\lambda_1, \dots,
\lambda_{u_1} \in \RR$  и $D_j$ -- матрицы $2 \times 2$, имеющие вид
$$
D_j = \left(
        \begin{array}{cc}
          a_j & -b_j \\
          b_j & a_j \\
        \end{array}
      \right), \quad \mbox{где $a_j > 0$ и $b_j \ne 0$, $j \in \{1, \dots, u_2\}$}.
$$

Аналогично, $DX(q) = \diag(A_q, B_q)$, где $Re(\lambda_j) > 0$ для
собственных чисел $A_q$; $ Re(\lambda_j) < 0$ для собственных чисел
$B_q$ и $ B_q = \diag (\mu_1, \dots, \mu_{s_1}, \tilde{D}_1, \dots,
\tilde{D}_{s_2}), $ где $\mu_1, \dots, \mu_{s_1} \in \RR$ и
$\tilde{D}_j$ -- матрицы $2 \times 2$, имеющие вид
$$
\tilde{D}_j = \left(
        \begin{array}{cc}
          \tilde{a}_j & -\tilde{b}_j \\
          \tilde{b}_j & \tilde{a}_j \\
        \end{array}
      \right), \quad \mbox{где $\tilde{a}_j < 0$ и $\tilde{b}_j \ne 0$, $j \in \{1, \dots, s_2\}$}.
$$
Таким образом, в этих окрестностях $N_p$ и $N_q$ (а в дальнейшем мы
будем считать, что все рассмотрение ведется в объединении $N_p$ и
$N_q$ и малой окрестности траектории точки $r$) верно следующее:
$$
\Ws(p) = \{ z = 0 \}, \; \Wu(p) = \{ y = 0 \},
 \;
\Ws(q) = \{ \eta = 0 \},
 \;
\Wu(q) = \{ \xi = 0 \}.
$$

Введем обозначения: $ S_p = \Ws(p)$, $U_p = \Wu(p)$, $S_q = \Ws(q)$,
$U_q = \Ws(q)$. Пусть при этом $S_q = S_{q}^{(1)} \oplus \dots
\oplus S_{q}^{(l)}$, где $l = s_1 + s_2$ и $S_{q}^{(1)}, \dots,
S_{q}^{(l)}$ -- одномерные или двумерные инвариантные относительно
$DX(q)$ подпространства. Аналогично, $ U_p = U_p^{(1)} \oplus \dots
\oplus U_p^{(m)}$, где $m = u_1 + u_2$ и $U_p^{(1)}, \dots,
U_p^{(m)}$~-- одномерные или двумерные инвариантные относительно
$DX(p)$ подпространства.

Обозначим для $j = 1, \dots, l$ через $\Pi_q^{(j)}$ проекторы на
$S_q^{(j)}$ параллельно $U_q \oplus S_{q}^{(1)} \oplus \dots \oplus
S_{q}^{(j-1)} \oplus S_{q}^{(j + 1)} \oplus \dots \oplus
S_{q}^{(l)}$. При этом будут выполнены равенства
$$
\Pi_q^{(j)}S_q^{(j)} = S_q^{(j)};  \quad \Pi_q^{(j)} \Pi_q^{(k)} =
0, \quad \mbox{где} \quad j, k = 1, \dots, l, \quad j \ne k.
$$
Обозначим через $\Pi_q$ проектор на $S_q$ параллельно $U_q$: $
{\Pi_q = \Pi_q^{(1)} + \dots + \Pi_q^{(l)}}$.

Пусть $\Pi_p^{(1)}, \dots, \Pi_p^{(m)}$ -- проекторы на $U_p^{(1)},
\dots, U_p^{(m)}$ соответственно. Верно следующее:
$$
\Pi_p^{(i)}U_p^{(i)} = U_p^{(i)}, \quad \Pi_p^{(i)} \Pi_p^{(k)} = 0,
\quad \mbox{где}\quad  i, k = 1, \dots, m, \quad i \ne k.
$$
Обозначим через $\Pi_p$ проектор на $U_p$ параллельно $S_p$: $
{\Pi_p = \Pi_p^{(1)} + \dots + \Pi_p^{(m)}} $.

Выберем на траектории $\phi(t, r)$ точки $a_p \in N_p$, $a_q \in
N_q$ таким образом, чтобы для любого $t>0$ выполнялись включения $
\phi(t, a_p) \in N_p$, ${\phi(-t, a_q) \in N_q}$. При этом для
некоторого $\tau > 0$ будет выполнено равенство $ a_p = \phi(\tau,
a_q). $ Пусть $v_p = X(a_p)$, $v_q = X(a_q)$. Ясно, что $v_p \in
S_p$, $v_q \in U_q$.

Пусть $\tilde{\Sigma}_p$ -- гиперплоскость в $S_p$, ортогональная
$v_p$, а $\Sigma_p$ -- аффинное $(n-1)$-мерное подпространство $
\Sigma_p = a_p + \tilde{\Sigma}_p + U_p. $ Аналогично,
$\tilde{\Sigma}_q$ -- гиперплоскость в $U_q$, ортогональная $v_q$, а
$\Sigma_q$ -- аффинное $(n-1)$-мерное подпространство $ \Sigma_q =
a_q + \tilde{\Sigma}_q + S_q. $ Ясно, что $\Sigma_p$ и $\Sigma_q$ не
имеют контакта с полем $X$ в малых окрестностях точек $a_p$ и $a_q$.
Обозначим через $ K: \Sigma_q \to \Sigma_p $ соответствующее
отображение Пуанкаре.

Возмущением поля $X$ и выбором координат около куска траектории
$\phi([o, \tau], a_q)$ можно добиться того, чтобы
\begin{enumerate}
\item [--] выполнялось
равенство $K(x) = \phi(\tau, x)$ для $x \in \Sigma_q$, близких к
$a_q$;
\item [--] отображение $K$ было линейным (при естественном
отождествлении $\Sigma_q$ с $\tilde{\Sigma}_q \oplus S_q$ и
$\Sigma_p$ с $\tilde{\Sigma}_p \oplus U_p$).
\end{enumerate}
Ясно, что в этом случае
\begin{equation}\label{Lip*17.1}
T_{a_p}\Wu(q) = K \tilde{\Sigma}_q + v_p \quad \mbox{и} \quad
T_{a_p}\Ws(p) = \tilde{\Sigma}_p + v_p.
\end{equation}
Нетрансверсальность пересечения $\Wu(q)$ и $\Ws(p)$ в точке $a_p$
означает, что $T_{a_p}\Wu(q) + T_{a_p}\Ws(p) \ne \RR^n$. Ввиду
соотношений \sref{Lip*17.1} это означает, что $ {v_p +
\tilde{\Sigma}_p + K \tilde{\Sigma}_q \ne \RR^n}$. Отсюда, ввиду
равенства $v_p + \tilde{\Sigma}_p = S_p$, следует, что
\begin{equation}\label{Lip*17.3}
\Pi_p K \tilde{\Sigma}_q \ne U_p.
\end{equation}
Из соотношения \sref{Lip*17.3} следует, что при некотором $i \in
\{1, \dots, m\}$ выполнено соотношение $\Pi_p^{(i)}K
\tilde{\Sigma}_q \ne U_p^{(i)}$. Мы рассмотрим наиболее сложный
случай, в котором $\dim U_p^{(i)} = 2$ и $\dim \Pi_p^{(i)}K
\tilde{\Sigma}_q = 1$. Обозначим через $e_p \in U_p^{(i)}$ единичный
вектор, перпендикулярный $\Pi_p^{(i)}K \tilde{\Sigma}_q$. Обозначим
через $\Pi_p^{e_p}$ проекцию на прямую, проходящую через вектор
$e_p$, параллельно $\Pi_p^{(i)}K \tilde{\Sigma}_q$.  Из выбора $e_p$
следует равенство
\begin{equation}\label{Text3.5}
\Pi_p^{e_p}K \tilde{\Sigma}_q = \{0\}.
\end{equation}
Любой вектор $x \in \Sigma_q$ представим в виде $x = \Pi_qx +y$, где
$y \in \tilde{\Sigma}_q$. Отсюда следует, что $\Pi_p^{e_p} Kx =
\Pi_p^{e_p} K (\Pi_qx +y) = \Pi_p^{e_p}K\Pi_q x + \Pi_p^{e_p}Ky$. Из
равенства \sref{Text3.5} следует, что
\begin{equation}\label{Add12.0.5}
\Pi_p^{e_p} Kx = \Pi_p^{e_p}K\Pi_q x \quad \mbox{для $x \in
\Sigma_q$}.
\end{equation}
Ввиду того, что $\Sigma_p = K\Sigma_q = K (\tilde{\Sigma}_q + S_q)$,
выполнено соотношение
\begin{equation}\label{Add12.1}
\Pi_p^{e_p} K S_q \ne \{0\}.
\end{equation}
В дальнейшем мы будем ссылаться лишь на соотношения \sref{Add12.0.5}
и \sref{Add12.1}; разбор других случаев отличается лишь выбором
вектора $e_p$.

Отождествим прямую, проходящую через $e_p$, с вещественной прямой и
будем считать, что $\Pi_p^{e_p} e_p = 1$. Выберем такой единичный
вектор $e_q \in S_q$, чтобы для всех $j \in \{1, \dots, l\}$
выполнялись соотношения:
\begin{enumerate}
\item $\Pi_q^{(j)} e_q = 0$, если $\Pi_p^{e_p} K S_q^{(j)} = \{0\}$.
\item $\Pi_p^{e_p} K \Pi_q^{(j)} e_q < 0$, если $\Pi_p^{e_p} K S_q^{(j)} \ne
\{0\}$, при этом если $\dim S_q^{(j)} = 2$, то $e_q$ выберем таким
образом, чтобы $\Pi_q^{(j)} e_q \perp \Ker \Pi_p^{e_p} K
\Pi_q^{(j)}$.
\end{enumerate}

Из \sref{Add12.1} следует, что существует $e_q \ne 0$. Для всякого
$d > 0$ рассмотрим псевдотраекторию $g(t)$ следующего вида:
$$
g(t) =
\begin{cases}
\phi(t, a_q +de_q), & t < 0,\\
\phi(t, a_q),& 0 \leq t < \tau,\\
\phi(t, a_p +de_p),& t \geq \tau.
\end{cases}
$$
Ясно, что существует такая константа $C_1 \geq 1$, зависящая лишь от
потока $\phi$ и не зависящая от выбора $d, e_p, e_q$, что $g(t)$
будет $C_1 d$-псевдотраекторией потока $\phi$.

Предположим, что поле $X$ обладает липшицевым свойством отслеживания
с константами $L_0$ и $D_0$. Пусть, согласно нашему предположению,
псевдотраектория $g(t)$ при $D_0/C_1>d>0$ отслеживается траекторией
точки $w_q$ с репараметризацией $h(t) \in \Rep(L_0 C_1 d)$. При этом
выполнено неравенство
\begin{equation}\label{Lip38}
\dist(\phi(h(t), w_q), g(t)) \leq L_0 C_1 d, \quad t \in \RR.
\end{equation}

Ясно, что траектория точки $\omega_q$ пересекает $\Sigma_q$.
Обозначим точку пересечения через $\omega'_q$. Из неравенства
\sref{Lip38} следует, что найдется такая константа $C_2$, не
зависящая от $d$, что $\omega'_q = \phi(H, \omega_q)$ при некотором
$|H| < C_2d$. Траектория точки $\omega'_q$ будет отслеживать
псевдотраекторию $g(t)$ с репараметризацией класса $\Rep(L' C_1 d)$,
где $L' = (L_0C_1+C_2)/C_1$. Для простоты дальнейшего изложения
переобозначим $\omega'_q$ через $\omega_q$ и $L'$ через $L_0$.

Определим $w_p \in \Sigma_p$ следующим образом: $ w_p = K w_q =
\phi(\tau, w_q).$ Из включения $g(\tau) \in \Sigma_p$ и из
неравенства \sref{Lip38} при $t = \tau$ следует, что $
\dist(\phi(h(\tau), w_q), \Sigma_p) \leq L_0 C_1 d$. Ясно, что в
этом случае существует такая константа $C_3$, не зависящая от $d$,
что $ w_p = \phi(h(\tau)+H, w_q)$ при некотором $|H| < C_3d$.

Пусть $ \phi_q^{(j)}(t, x) = \Pi_q^{(j)}\phi(t, \Pi_q^{(j)}x)$ --
проекция потока $\phi$ на подпространство $S_q^{(j)}$. Ясно, что
$\phi_q^{(j)}$ задается линейным векторным полем до тех пор, пока
траектория не покидает окрестность $N_q$. Аналогично вводится $
\phi_p(t, x) = \Pi_p^{(i)}\phi(t, \Pi_p^{(i)}x)$.

Возьмем $\varepsilon = \pi/4$ и $L = C_1L_0 + 1$. Применим к этим
числам и потокам $\phi_p(t, x)$ и $\phi_q^{(j)}(-t, x)$ при $j \in
\{1, \dots, l\}$ леммы \ref{Spiral} и \ref{Dim1Spiral}. Найдем такие
числа $T = T(\varepsilon, L)$ \label{TSelect} и $d_0 =
d_0(\varepsilon, L)$, что утверждение лемм \ref{Spiral} и
\ref{Dim1Spiral} будет выполнено для данных $T$ и $d_0$ для всех
рассматриваемых систем.

Выберем такое $d_1 \in \RR$, что $d_0 > d_1>0$ и для любого $d \leq
d_1$ выполнены неравенство \sref{Lip38} и включения
\begin{equation}\notag
B(L_0C_1d, \phi(t, a_q + d e_q)) \subset N_q \quad \mbox{при} \; 0
\geq t \geq -2T
\end{equation}
и
\begin{equation}\notag
B(L_0C_1d, \phi(t, a_p + d e_p)) \subset N_p \quad \mbox{при} \; 0
\leq t \leq 2T,
\end{equation}
где $B(a, x)$ -- шар радиуса $a$ с центром в точке $x$.

Отсюда и из неравенства \sref{Lip38} следует, что при $0 \leq t \leq
T$ выполнены включения
$$
\phi(h(-t), \omega_q) \in N_q \quad \mbox{и} \quad \phi(h(\tau + t)
-h(\tau), \omega_p) \in N_p.
$$
Таким образом, интересующие нас куски траектории и псевдотраектории
лежат в $N_p$ и $N_q$.

Из неравенств \sref{Lip38} и из определения $g(t)$ следует, что
$$
|\phi_q^{(j)}(h(t), w_q) - \phi_q^{(j)}(t, a_q + d e_q)| \leq L_0
C_1 d \; \mbox{при} \; -T \leq t \leq 0, j \in \{1, \dots, l\}.
$$

Покажем, что $\Pi_p^{e_p}K \Pi_q^{(j)}d e_q$ и
$\Pi_p^{e_p}K\Pi_q^{(j)} \omega_q$ одного знака. Рассмотрим более
сложный случай, когда $\dim S_q^{(j)} = 2$. Применим лемму
\ref{Spiral} к потоку $\phi_q^{(j)}(-t, x)$ с $x_0^{(j)} =
\Pi_q^{(j)}(d e_q)$ и $x_1^{(j)} = \Pi_q^{(j)} \omega_q$. Мы видим,
что $ |\arg(x_1^{(j)}) - \arg(x_0^{(j)})| < \varepsilon =
\frac{\pi}{4}$. Из выбора $e_q$ следует, что $e_q$ и $\omega_q$
лежат в одной полуплоскости относительно $\Ker \Pi_p^{e_p} K
\Pi_q^{(j)}$. Отсюда следует, что $\Pi_p^{e_p}K \Pi_q^{(j)} e_q$ и
$\Pi_p^{e_p}K\Pi_q^{(j)} \omega_q$ одного знака, т.е.
\begin{equation}\label{Text19.5}
\Pi_p^{e_p}K\Pi_q^{(j)} \omega_q < 0.
\end{equation}
Аналогично показывается, что $\Pi_p^{e_p}\omega_p$ и
$\Pi_p^{e_p}de_p$ одного знака, т.е. ${\Pi_p^{e_p}\omega_p >0}$.
Складывая неравенства \sref{Text19.5} для всех $j \in \{1, \dots,
l\}$, получим неравенство $\Pi_p^{e_p}K \Pi_q \omega_q < 0$. Из
\sref{Add12.0.5} следует, что $\Pi_p^{e_p}K \omega_q < 0$, однако
$\Pi_p^{e_p}K \omega_q = \Pi_p^{e_p} {\omega_p > 0}$. Это
противоречие доказывает лемму \ref{Rest} и теорему~\ref{Lip}.
\end{proof}

\section{Доказательство теоремы \ref{Dim3}}

Для доказательства теоремы \ref{Dim3} нам понадобится две
дополнительные леммы.
\begin{lem}\label{NoSubset}
Пусть $p$ и $q$ -- гиперболические точки покоя векторного поля $X$,
при этом точка $p$ не является стоком. Пусть $ r = \Wu(q) \cap
\Ws(p).$ Предположим, что в некоторой окрестности $V$ точки $r$
выполнено включение
\begin{equation}\label{Sovp1.1}
\Wu(q) \cap V \subset \Ws(p) \cap V.
\end{equation}
Тогда $X \notin \Int^1(\OrSh)$.
\end{lem}
\begin{proof}
Не умаляя общности, можно считать, что $r \in W_{loc}^s(p)$ (где
$W_{loc}^s(p)$ и $W_{loc}^u(p)$ -- соответственно локально
устойчивое и локально неустойчивое многообразие точки $p$).
Рассмотрим произвольную точку~${\alpha \in W_{loc}^u(p)}$. Выберем
такое $\varepsilon > 0$, что
\begin{enumerate}
\item $\dist(\alpha, \Ws_{loc}(p)) > \varepsilon$ и $B(\varepsilon, r) \subset V$;
\item траектория любой точки $x \notin \Wu(q)$ при стремлении времени к $-\infty$
покидала бы $\varepsilon$-окрестность точки $q$.
\end{enumerate}
Для произвольных $\tau_0, \tau_1 > 0$ рассмотрим псевдотраекторию
$g(t)$ следующего вида:
$$
g(t) =
\begin{cases}
\phi(t, r),& \quad t \leq \tau_0, \\
\phi(t - \tau_0 - \tau_1, \alpha),& \quad t > \tau_0.
\end{cases}
$$
Поскольку $\phi(t, r) \to p$ при $t \to \infty$ и $\phi(t, \alpha)
\to p$ при $t \to -\infty$, то для всякого $d > 0$ найдутся такие
$\tau_0$ и $\tau_1$, что $g(t)$ будет являться
{$d$-псевдотраекторией}.

Покажем, что для любой репараметризации $h(t)$ и точки $x \in M$
найдется такое $t \in \RR$, что $ \dist(g(t), \phi(h(t), x)) >
\varepsilon$. Предположим противное, тогда выполнено неравенство
\begin{equation}\label{Sovp3.2}
\dist(g(t), \phi(h(t), x)) \leq \varepsilon, \quad t \in \RR.
\end{equation}
Поскольку $g(t) \to q$ при $t \to -\infty$, то из неравенства
\sref{Sovp3.2} следует, что ${x \in \Wu(q)}$. Подставив $t = 0$ в
\sref{Sovp3.2}, получим неравенство $\dist(r, \phi(h(0), x)) \leq
\varepsilon$. Отсюда и из соотношения \sref{Sovp1.1} следует, что
$\phi(h(0), x) \in W^s_{loc} (p)$, а значит для любого $t>0$
выполнено включение $ \phi(h(t), x) \in W^s_{loc}(p). $ Но тогда,
исходя из выбора $\alpha$, при $t = \tau_0+\tau_1$ не выполнено
неравенство \sref{Sovp3.2}. А отсюда следует, что $X \notin
\Int^1(\OrSh)$.
\end{proof}
\begin{lem}\label{Dim1} Пусть $p, q$ -- гиперболические точки покоя
векторного поля $X \in \Int^1(\OrSh)$ и при этом $\dim \Wu(p) = 1$.
Пусть $r \in \Wu(q) \cap \Ws(p)$. Тогда $r$ -- точка
трансверсального пересечения $\Wu(q)$ и $\Ws(p)$.
\end{lem}
\begin{proof}
Предположим, что $r$ -- точка нетрансверсального пересечения
$\Wu(q)$ и $\Ws(p)$. По аналогии с доказательством леммы \ref{Rest}
можно считать, что
\begin{enumerate}
\item поле $X$ линейно в некоторой окрестности $U$ точки $p$ и при этом $r \in U$;
\item в некоторой окрестности $V$ точки $r$ многообразие $\Wu(q)$ имеет
вид $r+K$, где $K$ -- некоторое линейное подпространство.
\end{enumerate}

Поскольку $\Wu(q)$ и $\Ws(p)$ в окрестности $V$ представляют собой
аффинные пространства и при этом $\dim \Ws(p) = \dim M - 1$, то из
нетрансверсальности пересечения $\Wu(q)$ и $\Ws(p)$ следует, что $
\Wu(q) \cap V \subset \Ws(p) \cap V$. Из этого соотношения и из
леммы \ref{NoSubset} следует, что $X \notin \Int^1(\OrSh)$.
\end{proof}

\begin{proof}[Доказательство теоремы \ref{Dim3}]
Рассмотрим многообразие $M$ размерности $\dim M \leq 3$. По аналогии
с доказательством теоремы \ref{Lip} нам достаточно доказать, что
если $X \in \Int^1(\OrSh)$, $p, q \in \Per(X)$ и $r \in \Wu(q) \cap
\Ws(p)$, то $r$ -- точка трансверсального пересечения $\Ws(p)$ и
$\Wu(q)$. Если $p$ или $q$ является замкнутой траекторией, то данное
утверждение следует из леммы~\ref{OrSP}. Таким образом, можно
считать, что $p$ и $q$ -- точки покоя. Предположим, что $\dim M = 3$
(доказательство для случаев $\dim M = 2$ и  $\dim M = 1$
аналогично). Возможны следующие случаи:
\begin{enumerate}
\item Хотя бы одно из многообразий $\Ws(p)$ и $\Wu(q)$ имеет
размерность~3. Тогда их пересечение является трансверсальным.
\item Хотя бы одно из многообразий $\Ws(p)$ и $\Wu(q)$ имеет
размерность~2. Не умаляя общности, можно считать, что $\dim \Ws(p) =
2$. Тогда из леммы \ref{Dim1} следует, что пересечение $\Ws(p)$ и
$\Wu(q)$ трансверсально.
\item Оба многообразия $\Ws(p)$ и $\Wu(q)$ имеют размерность 1. В
этом случае каждое из них представляет из себя траекторию некоторой
точки, а значит, $\Ws(p) = \Wu(q)$. Из леммы \ref{NoSubset} следует,
что $X \notin \Int^1(\OrSh)$.
\end{enumerate}
Теорема \ref{Dim3} доказана.
\end{proof}

\section{Заключение}

В статье описана структура $\Cone$-внутренности множеств векторных
полей, обладающих липшицевым и ориентированным свойствами
отслеживания.

\end{document}